\documentclass[letterpaper, 10 pt, conference]{ieeeconf}  % Comment this line out if you need a4paper

\IEEEoverridecommandlockouts                              % This command is only needed if 
\usepackage{amssymb}
\usepackage{amsmath}
\usepackage{bm}
\usepackage{graphicx}
\usepackage{ mathrsfs }
\newtheorem{theorem}{Theorem}

\newtheorem{lemma}{Lemma}
\newtheorem{assumption}{Assumption}

\usepackage{color}
\usepackage{cite}
\usepackage{textcmds}

\allowdisplaybreaks

\title{\LARGE \bf
Unbiased Extremum Seeking for PDEs
}

\author{Cemal Tugrul Yilmaz, Mamadou Diagne and Miroslav Krstic
\thanks{C. T. Yilmaz, M. Diagne and M. Krstic are with Department of Mechanical and Aerospace Engineering, University of California, San Diego, La Jolla, CA, USA.
        {\tt\small cyilmaz@ucsd.edu, mdiagne@ucsd.edu, krstic@ucsd.edu}}
}

\begin{document}

\maketitle
\thispagestyle{empty}
\pagestyle{empty}

%%%%%%%%%%%%%%%%%%%%%%%%%%%%%%%%%%%%%%%%%%%%%%%%%%%%%%%%%%%%%%%%%%%%%%%%%%%%%%%%
\begin{abstract}

There have been recent efforts that combine seemingly disparate methods, extremum seeking (ES) optimization 
and partial differential equation (PDE) backstepping, to address the problem of model-free optimization with 
PDE actuator dynamics. In contrast to prior PDE-compensating ES designs, which only guarantee local stability 
around the extremum, we introduce unbiased ES that compensates for delay and diffusion PDE dynamics while 
ensuring exponential and unbiased convergence to the optimum. Our method leverages exponentially decaying/growing signals 
within the modulation/demodulation stages and carefully selected design parameters. The stability analysis 
of our designs relies on a state transformation, infinite-dimensional averaging, local exponential stability 
of the averaged system, local stability of the transformed system, and local exponential stability of the 
original system. Numerical simulations are presented to demonstrate the efficacy of the developed designs.

\end{abstract}

%%%%%%%%%%%%%%%%%%%%%%%%%%%%%%%%%%%%%%%%%%%%%%%%%%%%%%%%%%%%%%%%%%%%%%%%%%%%%%%%
\section{Introduction} \label{intro}
In the context of real-time optimization, extremum seeking (ES) emerges as an effective optimization technique with a century-old history \cite{scheinker2024100}, originating from the work of \cite{leblanc1922electrification}. Its popularity continues to grow over time because it is easy to use and does not require a detailed model of the cost function. The expanding theoretical groundwork in ES has led to diverse applications, ranging  from exoskeletons \cite{kumar2020extremum} to quantum computers \cite{abbasgholinejad2023extremum}.
However, much of the literature on ES primarily focuses on optimizing systems described by ordinary differential equations (ODEs). However, the model of many physical systems, such as drilling systems, reactors, batteries, and continuum robots, are more complex and these systems are better described by partial differential equations (PDEs).

There exist various optimization challenges in infinite dimensional systems. In certain systems like network control systems, and cyber-physical systems, time delays between control action and system response are inevitable. This issue requires a delay-aware ES, as designs that ignore the delay may lead to instability, especially with large delays. In oil drilling systems \cite{aarsnes2019extremum}, the input dynamics involve a cascade of wave PDE and ODE, and the objective is to maximize the rate of penetration (ROP) up to a certain threshold known as the foundering point. Beyond this point, ROP starts decreasing, leading to energy wastage and potential cutter damage. In tubular reactors \cite{hudon2008adaptive}, which are defined by coupled hyperbolic PDEs, the goal is to seek a reactor temperature profile maximizing the reactor exit concentration. Another intriguing challenge arises in pool boiling systems \cite{alessandri2020stabilization}. As heat flux increases during boiling, bubbles form and rise to the surface. Beyond a critical heat flux, bubbles cease to rise, and a vapor film covers the heater surface, acting as an insulator. This leads to a significant temperature increase above the heater material's melting point, causing physical burnout of the heater. ES can be adapted to stabilize the heat flux at the unknown optimal level. 

It is important to note that existing ES designs \cite{aarsnes2019extremum, hudon2008adaptive} focus on optimization at the steady state and do not fully account for the PDE dynamics. This limits their applicability to PDE systems with slow transient and motivates the development of PDE compensated ES designs. Recent efforts have addressed this gap, focusing on ES designs for functions with input/output delays \cite{oliveira2016extremum} or those with inputs governed by diffusion PDEs \cite{feiling2018gradient}, wave PDEs \cite{oliveira2021extremum}, and PDE-PDE cascades \cite{oliveira2021extremum2}. For a more comprehensive treatment of the problem, refer to the monograph \cite{oliveira2022extremum}. However, PDE-compensating ES designs in \cite{feiling2018gradient, oliveira2021extremum, oliveira2021extremum2, oliveira2022extremum} can only ensure convergence to a neighborhood of the optimum point due to the active perturbing signal in the design, resulting in suboptimal performance.

This paper is an extension of our earlier work \cite{yilmaz2023exponential}, which introduces the unbiased extremum seeker (uES) for maps without PDE dynamics, the first ES design demonstrating exponential and unbiased convergence to an unknown optimum at a user-defined rate. Expanding upon \cite{yilmaz2023exponential}, we introduce ES designs achieving prescribed-time and unbiased convergence to the optimum \cite{yilmaz2023press}, and perfectly tracking time-varying optimum \cite{yilmaz2024perfect}.
In this paper, we present two distinct uES designs: one can handle arbitrarily long and known time delays, while the other compensates for diffusion PDEs. The designs consist of a PDE compensator, a perturbation signal with exponentially decaying amplitude (to eliminate steady-state oscillation), demodulation signals with exponentially growing amplitude and properly selected design parameters (to ensure unbiased convergence).

This paper is structured as follows. Section \ref{probsta} presents problem formulation. Section \ref{unbESdelay} and Section \ref{unbESdiff} introduce unbiased ES with delay compensator and unbiased ES with diffusion PDE compensator, respectively, along with formal stability analysis. Numerical results are presented in Section \ref{numsim}, followed by a conclusion in Section \ref{conclusion}.

\textit{Notation:} We denote the 
Euclidean norm by $|\cdot|$. The partial derivatives of a function $u(x,t)$ are denoted by $\partial_x u(x,t)=\partial u(x,t)/ \partial x$, $\partial_t u(x,t)=\partial u(x,t)/ \partial t$. 
The spatial $L_2[0,D]$ norm of $u(x,t)$ is denoted by $||u(\cdot, t)||^2=\int_0^D u^2(x,t)dx$.

\section{Problem Statement} \label{probsta}
We consider the optimization problem given by
\begin{align}
    \min_{\theta \in \mathbb{R}} Q(\theta), \label{maxQ}
\end{align}
where $\theta \in \mathbb{R}$ represents the input, $Q \in \mathbb{R} \to \mathbb{R}$ is an unknown smooth function. We introduce the following assumption regarding the unknown static map $Q(\cdot)$.

\begin{assumption} \label{assquad}
The unknown static map is characterized by the following quadratic form
\begin{align}
    Q(\theta)={}y^*+\frac{H}{2}(\theta-\theta^*)^2, \label{quadQ}
\end{align}
where $y^* \in \mathbb{R}$ and $\theta^* \in \mathbb{R}$ denote the unknown optimal output and input values, respectively, and $H > 0$ represents the unknown Hessian of the static map $Q(\theta)$.
\end{assumption}

In \cite{yilmaz2023exponential}, we design a multivariable ES that achieves unbiased and exponential convergence to the optimum,
assuming the absence of actuator dynamics. This work addresses a more challenging and application-relevant optimization problem by considering the same objective function \eqref{maxQ} but incorporating actuator dynamics modeled by a delay and diffusion PDE.
Our current objective is to develop a scalar unbiased ES that effectively compensates for PDE dynamics and exponentially guides the input $\theta(t)$ towards the optimum $\theta^*$.

\section{Unbiased ES with Delay} \label{unbESdelay}
\begin{figure}[t]
    \centering
    \includegraphics[width=\columnwidth]{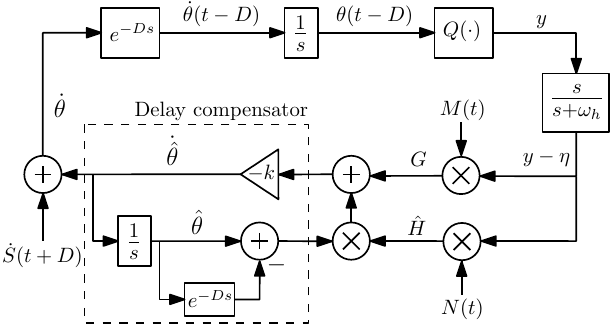} 
    \caption{Unbiased ES with delay compensator. The design employs exponentially growing multiplicative  signals, $M(t)=\frac{2}{a}e^{\lambda t}\sin(\omega t)$, and $N(t)=-\frac{8}{a^2}e^{2\lambda t}\cos(2\omega t)$, as well as  exponentially decaying dither signal $\dot{S}(t+D)=\frac{d}{dt}\left(e^{-\lambda (t+D)} a  \sin(\omega (t+D))\right)$ to achieve exponential and unbiased convergence to the optimum $\theta^*$ at the rate of the user-defined $\lambda>0$. }
    \label{ESDeBlock}
\end{figure}
In this section, we consider the scenario where the output $y(t)$ is subject to a known and constant delay $D \in \mathbb{R}$, expressed as
\begin{align}
    y(t)=Q(\theta(t-D)), \quad t \in [0, \infty). \label{output}
\end{align}
Note that the static map $Q(\theta)$ enables the representation of the overall delay $D$ as the sum of individual components, denoted as $D_u$ and $D_y$, corresponding to the delays in the actuation and measurement paths, respectively. 
Fig. \ref{ESDeBlock} illustrates the closed-loop unbiased ES with delay compensator. Before delving into the estimator design, we introduce the signals presented in Fig. \ref{ESDeBlock} in the following subsection.

\subsection{Excitation signals and gradient/Hessian estimates}
Let us define the following parameter estimate
\begin{align}
    \hat{\theta}(t)=\theta(t)-S(t+D), \label{thetahatdef}
\end{align}
with the perturbation signal 
\begin{align}
    S(t+D)={}e^{-\lambda (t+D)} a  \sin(\omega (t+D)),  \label{Sdef}
\end{align}
where $\lambda>0$ is the decay rate of the perturbation signal, $a \in \mathbb{R}$ is the perturbation amplitude, $\omega >0$ is the probing frequency. Let us define the delay-free and delayed parameter estimation error variables
\begin{align}
    \tilde{\theta}(t)={}&\hat{\theta}(t)-\theta^*, \quad \tilde{\theta}(t-D)={}\hat{\theta}(t-D)-\theta^*. \label{delfdelayed}
\end{align}
Applying the technique in \cite{krstic2009delay}, we represent the signals \eqref{delfdelayed} through the transport PDE as
\begin{align}
    \tilde{\theta}(t-D)={}&\bar{u}(0,t), \label{transPDE1} \\
    \partial_t \bar{u}(x,t)={}&\partial_x \bar{u}(x,t),    \\
    \bar{u}(D,t)={}&\tilde{\theta}(t) \label{transPDE3}
\end{align}
for $ x\in (0, D)$. The solution of this PDE is given by
\begin{align}
    \bar{u}(x, t)={}&\tilde{\theta}(t+x-D). \label{ubarsol}
\end{align}
We compute the estimate of the gradient and Hessian as follows
\begin{align}
    G(t)={}&M(t)(y(t)-\eta(t)), \label{Gdef} \\
    \hat{H}(t)={}&N(t)(y(t)-\eta(t)), \label{Hdef}
\end{align}
where the multiplicative excitation signals are given by
\begin{align}
    M(t)={}&\frac{2}{a} e^{\lambda t}\sin(\omega t), \quad  
    N(t)=-\frac{8}{a^2} e^{2\lambda t} \cos(2\omega t)\label{MNdef}
\end{align}
and $\eta(t)$ is governed by
\begin{align}
    \dot{\eta}(t)={}-\omega_h \eta(t)+\omega_h y(t). \label{etadyn}
\end{align}
We use exponentially decaying signals in the modulation stage, as defined in \eqref{Sdef}, and exponentially growing signals in the demodulation stage, as defined in \eqref{Gdef} and \eqref{Hdef}. The important aspect of our design lies in the fact that the high-pass filtered state $y-\eta$ exhibits exponential decay to zero at the rate of $2\lambda$. 
This crucial feature ensures that despite the growing amplitudes in \eqref{MNdef}, the estimates in \eqref{Gdef} and \eqref{Hdef} remain bounded.

\subsection{ Parameter estimator design and error dynamics}
The time derivative of \eqref{transPDE1}--\eqref{transPDE3} is given by
\begin{align}
    \dot{\hat{\theta}}(t-D)={}&{u}(0,t), \label{transDPDE1} \\
    \partial_t {u}(x,t)={}&\partial_x {u}(x,t), \\
    {u}(D,t)={}&\dot{\hat{\theta}}(t), \label{transDPDE3}    
\end{align}
by noting $\dot{\tilde{\theta}}(t)=\dot{\hat{\theta}}(t)$ from \eqref{delfdelayed}. Following the methodology presented in \cite{krstic2009delay}, we consider the following backstepping transformation
\begin{align}
    w(x,t)={}&u(x,t)+\bar{k}\left(\tilde{\theta}(t-D)+\int_0^x u(\sigma,t)d\sigma \right) \label{tranfw}
\end{align}
to convert the system \eqref{transDPDE1}--\eqref{transDPDE3} into the target system
\begin{align}
    \dot{\tilde{\theta}}(t-D)={}&-\bar{k}\tilde{\theta}(t-D)+w(0,t), \label{targPDE1} \\
    \partial_t {w}(x,t)={}&\partial_x {w}(x,t), \\
    {w}(D,t)={}&0, \label{targPDE3}       
\end{align}
with the gain $\bar{k}>0$. 
By substituting $x=D$ into \eqref{tranfw} and considering \eqref{transDPDE3} and \eqref{targPDE3}, we obtain the update law
\begin{align}
    \dot{\hat{\theta}}(t)={}&-\bar{k}\left( \tilde{\theta}(t-D)+\int_0^D u(\sigma,t)d\sigma \right), \nonumber \\
    ={}&-\bar{k} \tilde{\theta}(t-D)-\bar{k} \left(\hat{\theta}(t)-\hat{\theta}(t-D) \right), \label{UpdLaw}
\end{align}
where we use the property $u(\sigma,t)=\partial_t \bar{u}(\sigma,t)=\partial_x \bar{u}(\sigma,t)$,  use the solution \eqref{ubarsol}, and recall \eqref{delfdelayed}.  
We obviously cannot use \eqref{UpdLaw} since the error $\tilde{\theta}(t-D)$ is not measured.
To overcome this limitation, we define $\bar{k}=kH$, where $H>0$ represents the unknown Hessian, and the user specifies the positive gain $k$. Then, we replace the signals $H \tilde{\theta}(t-D)$ and $H$ with their respective estimates, $G(t)$ from \eqref{Gdef} and $\hat{H}(t)$ from \eqref{Hdef}. This leads to the following implementable version of \eqref{UpdLaw}
\begin{align}
    \dot{\hat{\theta}}(t)={}&-kG(t)-k\hat{H}(t)\left(\hat{\theta}(t)-\hat{\theta}(t-D) \right). \label{delayuES}
\end{align}
The design parameters should satisfy the following conditions
\begin{align}
    \lambda < \frac{\omega_h}{2}, \qquad k>\frac{\lambda}{H}. \label{paramcond}
\end{align}
The essence of these conditions is that the adaptation (learning) rate should surpass the decay rate of the perturbation (exploration) signal.

In view of the transformation,
\begin{align}
    \tilde{\eta}(t)={}&\eta(t)-y^*, \label{etatilde}
\end{align}
we write the closed-loop system in the form
\begin{align}
    \tilde{\theta}(t-D)={}&\bar{u}(0,t), \label{closeddelay1} \\
    \partial_t \bar{u}(x,t)={}&\partial_x \bar{u}(x,t), \\
    \bar{u}(D,t)={}&\tilde{\theta}(t), \\
    \dot{\tilde{\theta}}(t)={}&-kG(t)-k\hat{H}(t)(\tilde{\theta}(t)-\tilde{\theta}(t-D)), \label{thtilclosedelay} \\
    \dot{\tilde{\eta}}(t)={}& -\omega_h \tilde{\eta}(t)+\omega_h  \left(y(t)-y^*\right), \label{closeddelay2}
\end{align}
where \eqref{Gdef} and \eqref{Hdef} are rewritten as
\begin{align}
    G(t)={}&\frac{2}{a}e^{\lambda t}\sin(\omega t)\left(y(t)-y^*-\tilde{\eta}(t)\right), \label{Grewrite} \\
    \hat{H}(t)={}&-\frac{8}{a^2}e^{2\lambda t} \cos(2\omega t)\left(y(t)-y^*-\tilde{\eta}(t)\right). \label{Hrewrite}
\end{align}
Recalling \eqref{quadQ}--\eqref{delfdelayed}, the output \eqref{output} is rewritten as
\begin{align}
    y(t)={}&y^*+\frac{H}{2}\left(\tilde{\theta}(t-D)+e^{-\lambda t}a\sin(\omega t)\right)^2. \label{outputdelay}
\end{align}

\subsection{Stability analysis}
The main theorem is stated as follows.
\begin{theorem} \label{theodelay}
Let Assumption \ref{assquad} holds and the parameters satisfy \eqref{paramcond}. Then,
there exists $\bar{\omega}$ and for any $\omega > \bar{\omega}$, the closed-loop system \eqref{closeddelay1}--\eqref{closeddelay2} is exponentially stable at the origin in the sense of the norm
\begin{align}
    \left( ||\bar{u}(\cdot, t)||^2+|\tilde{\theta}(t)|^2+|\tilde{\eta}(t)|^2 \right)^{1/2}. \label{normintheo}
\end{align}
Furthermore,  
the input $\theta(t)$ and output $y(t)$ exponentially converge to $\theta^*$ and $y^*$, respectively.
\end{theorem}
\begin{proof}
Let us proceed through the proof step by step.

\textbf{Step 1: State transformation.}  Let us consider the following transformations
\begin{align}
    \tilde{\theta}_f(t-D)={}&e^{\lambda(t-D)} \bar{u}(0,t), \label{statetransfdelay1} \\
    \bar{u}_f(x,t)={}&e^{\lambda(t+x-D)}\bar{u}(x,t), \\
    \tilde{\theta}_f(t)={}&e^{\lambda t}\tilde{\theta}(t), \label{thetildefdel} \\
    \tilde{\eta}_f(t)={}&e^{2\lambda t}\tilde{\eta}(t), \label{statetransfdelay2}
\end{align}
which transform \eqref{closeddelay1}--\eqref{closeddelay2} to the following system
\begin{align}
    \tilde{\theta}_f(t-D)={}&\bar{u}_f(0,t), \label{transfPDE1} \\
    \partial_t \bar{u}_f(x,t)={}&\partial_x \bar{u}_f(x,t), \\
    \bar{u}_f(D,t)={}&\tilde{\theta}_f(t), \\
    \dot{\tilde{\theta}}_f(t)={}&\lambda\tilde{\theta}_f(t) -k \frac{2}{a}\sin(\omega t)\frac{H}{2}\bigg[\Big(e^{\lambda D}\tilde{\theta}_f(t-D) \nonumber \\
    &+a \sin(\omega t)\Big)^2-\tilde{\eta}_f(t)\bigg]-k\left(-\frac{8}{a^2}\cos(2\omega t)\right) \nonumber \\
    &\times \frac{H}{2}\left[\left(e^{\lambda D}\tilde{\theta}_f(t-D)+a \sin(\omega t)\right)^2-\tilde{\eta}_f(t)\right] \nonumber \\
    &\times \left(\tilde{\theta}_f(t)-e^{\lambda D}\tilde{\theta}_f(t-D)\right), \\
    \dot{\tilde{\eta}}_f(t)={}& -(\omega_h-2\lambda) \tilde{\eta}_f(t)+\omega_h \frac{H}{2} \nonumber \\
    & \times \left(e^{\lambda D}\tilde{\theta}_f(t-D)+a \sin(\omega t)\right)^2, \label{transfPDE2}
\end{align}
in view of \eqref{Grewrite}--\eqref{outputdelay}.

\textbf{Step 2: Averaging operation.}
The average of the transformed system \eqref{transfPDE1}--\eqref{transfPDE2} over the period $\Pi=2\pi/\omega$ is given by
\begin{align}
    \tilde{\theta}^{\rm av}_f(t-D)={}&\bar{u}^{\rm av}_f(0,t), \\
    \partial_t \bar{u}^{\rm av}_{f}(x,t)={}&\partial_x \bar{u}^{\rm av}_{f}(x,t), \label{ubaravtx} \\
    \bar{u}^{\rm av}_f(D,t)={}&\tilde{\theta}^{\rm av}_f(t), \\
    \dot{\tilde{\theta}}^{\rm av}_f(t)={}&-(kH-\lambda)\tilde{\theta}_f^{\rm av}(t), \label{thefav}  \\
    \dot{\tilde{\eta}}^{\rm av}_f(t)={}&-(\omega_h-2\lambda) \tilde{\eta}^{\rm av}_f(t)+\omega_h \frac{H}{2} \nonumber \\
    &\times \left(e^{2\lambda D}(\tilde{\theta}^{\rm av}_f(t-D))^2+a^2/2\right), \label{etaavtildelay}
\end{align}
where $\tilde{\theta}^{\rm av}_f(t)$, $\bar{u}^{\rm av}_{f}(\cdot,t)$, and $\tilde{\eta}^{\rm av}_f(t)$ denote the average versions
of the states $\tilde{\theta}_f(t)$, $\bar{u}_{f}(\cdot,t)$, and $\tilde{\eta}_f(t)$, respectively.

\textbf{Step 3: Stability of average system.} 
The solution to \eqref{thefav} is given by 
   $ \tilde{\theta}_f^{\rm av}(t)=\tilde{\theta}_f^{\rm av}(0)e^{-(kH-\lambda)t}$.
Then, we write the solution to the PDE \eqref{ubaravtx} as 
$
    \bar{u}^{\rm av}_f(x,t)={}\tilde{\theta}_f^{\rm av}(0)e^{-(kH-\lambda)(t+x-D)}.
$
Thus, the $(\tilde{\theta}_f^{\rm av}, \bar{u}^{\rm av}_f)$-system is exponentially stable at the origin for $kH>\lambda$. 
Using this fact, it is trivial to show that $\tilde{\eta}_f^{\rm av}(t)$ of \eqref{etaavtildelay}  exponentially converges to $\frac{\omega_hHa^2}{4(\omega_h-2\lambda)}$ for $\omega_h>2\lambda$.

\textbf{Step 4: Invoking averaging theorem.} Applying the averaging theorem for infinite-dimensional systems \cite{hale1990averaging}, we establish that there exists $\bar{\omega}$ and for any $\omega > \bar{\omega}$, the transformed system \eqref{transfPDE1}--\eqref{transfPDE2} with states $(\bar{u}_f(\cdot, t), \tilde{\theta}_f(t), \tilde{\eta}_f(t))$ has a unique  exponentially stable periodic solution $(\bar{u}^{\Pi}_f(\cdot, t), \tilde{\theta}_f^{\Pi}(t), \tilde{\eta}_f^{\Pi}(t))$ of period $\Pi=2\pi/\omega$ and this solution satisfies
\begin{equation}
    \bigg(||\bar{u}^{\Pi}_f(\cdot, t)||^2+|\tilde{\theta}_f^{\Pi}(t)|^2+\bigg|\tilde{\eta}_f^{\Pi}(t)-\frac{\omega_h H a^2}{4(\omega_h-2\lambda)}\bigg|^2\bigg)^{\frac{1}{2}} \leq \mathcal{O}\Big(\frac{1}{\omega}\Big) \label{transO}
\end{equation}
for all $t \geq 0$. Considering \eqref{transO} and recalling the transformations \eqref{statetransfdelay1}--\eqref{statetransfdelay2}, we deduce that the original error system \eqref{closeddelay1}--\eqref{closeddelay2} with states $\bar{u}(\cdot,t), \tilde{\theta}(t), \tilde{\eta}(t)$ has a unique solution and is exponentially stable at the origin in the sense of the norm \eqref{normintheo}. 

\textbf{Step 5: Convergence to extremum.} 
Taking into account the results in Step 4 and recalling from \eqref{thetahatdef}--\eqref{delfdelayed}, \eqref{thetildefdel} that
\begin{align}
    \theta(t)={}e^{-\lambda t} \tilde{\theta}_f(t)+\theta^*+e^{-\lambda (t+D)} a  \sin(\omega (t+D)),
\end{align}
we conclude the exponential convergence of $\theta(t)$ to $\theta^*$. 
We establish the convergence of the output $y(t)$ to $y^*$ from \eqref{outputdelay} and complete the proof of Theorem \ref{theodelay}.
\end{proof}

\section{Unbiased ES with Diffusion PDE} \label{unbESdiff}
\begin{figure}[t]
    \centering
    \includegraphics[width=\columnwidth]{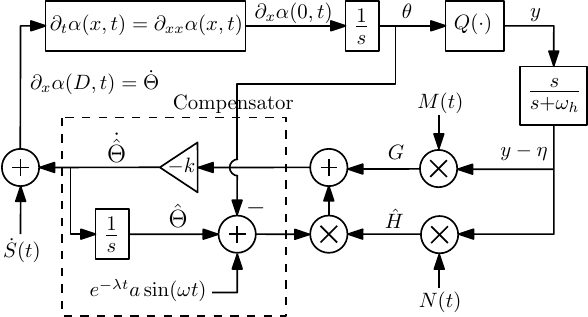} 
    \caption{Unbiased ES with diffusion PDE compensator. The design requires feedback of $\theta$, uses the same excitation signals, $M(t)$ and $N(t)$ as in Fig. \ref{ESDeBlock}, and employs a properly designed perturbation signal, $\dot{S}(t)$.}
    \label{ESheaBlock}
\end{figure}
In this section, we consider the following cascade of a diffusion PDE and ODE (integrator) with Neumann interconnection
\begin{align}
    \dot{\theta}(t)={}&\partial_x \alpha(0,t), \label{Thetadot} \\
    \partial_t \alpha(x,t)={}&\partial_{xx} \alpha(x,t), \\
    \alpha(0,t)={}&0, \\
    \partial_x \alpha(D,t)={}&\dot{\Theta}(t) \label{thetadot}
\end{align}
for $(x,t) \in (0, D) \times [0, \infty)$.
The output of the static map is 
\begin{align}
    y(t)={}y^*+\frac{H}{2}(\theta(t)-\theta^*)^2. \label{outheat}
\end{align}
The diffusion PDE with Neumann actuation arises in thermal systems, such as Stefan models of thermal phase change \cite{koga2020laser}, tubular reactors \cite{bovskovic2002backstepping}, and batteries \cite{moura2012pde}, where the control input is the heat flux. Our methodology can be extended to systems with Dirichlet actuation.
The unbiased ES with diffusion PDE compensator is schematically illustrated in Fig. \ref{ESheaBlock}.

\subsection{Perturbation signal}
Let us define the following parameter estimates to determine the optimal unknown actuator $\theta^*$
\begin{align}
    \hat{\theta}(t)={}&\theta(t)-e^{-\lambda t}a\sin(\omega t), \label{TheHat} \\
    \hat{\Theta}(t)={}&\Theta(t)-S(t).
\end{align}
We need to redesign the perturbation signal, ${S}(t)$,
such that when applied as input to the diffusion PDE
it produces the desired output, $e^{-\lambda t}a\sin(\omega t)$.
We formulate this trajectory generation problem as follows
\begin{align}
    S(t)={}&\partial_x \beta(D,t), \label{Sperturb} \\
    \partial_t \beta(x,t)={}&\partial_{xx} \beta(x,t), \label{betatxx} \\
    \beta(0,t)={}&0, \label{beta0} \\
    \partial_x \beta(0,t)={}&e^{-\lambda t} a\sin(\omega t) \label{betax0}
\end{align}
for $(x,t) \in (0, D) \times [0, \infty)$. We present the solution to $S(t)$ in the following lemma.
\begin{lemma} \label{lemS}
The explicit solution to $S(t)$ in \eqref{Sperturb} is
\begin{equation}
    S(t)=\frac{a}{2}e^{-\lambda t}\left(\sin(\omega t+qD)e^{pD}+\sin(\omega t-qD)e^{-pD}\right), \label{Sexplicit}
\end{equation}
where $p=\sqrt{\frac{\sqrt{\lambda^2+\omega^2}-\lambda}{2}}$ and $q=\sqrt{\frac{\sqrt{\lambda^2+\omega^2}+\lambda}{2}}$.
\end{lemma}

The proof of Lemma \ref{lemS} is given in Appendix A. 

Defining the error variables
\begin{align}
    \tilde{\theta}(t)={}&\hat{\theta}(t)-\theta^*, \qquad \tilde{\Theta}(t)={}\hat{\Theta}(t)-\theta^*, \label{errorsheat}
\end{align}
we write the following error system
\begin{align}
    \tilde{\theta}(t)={}&\partial_x \bar{u}(0,t), \label{errheat1} \\
    \partial_t \bar{u}(x,t)={}&\partial_{xx} \bar{u}(x,t), \\
    \bar{u}(0,t)={}&0, \\
    \partial_x \bar{u}(D,t)={}&\tilde{\Theta}(t), \label{errheat2}
\end{align}
where the time derivative is given by
\begin{align}
    \dot{\hat{\theta}}(t)={}&\partial_x {u}(0,t), \label{Thetahatdot} \\
    \partial_t u(x,t)={}&\partial_{xx} u(x,t), \\
    u(0,t)={}&0, \\
    \partial_x u(D,t)={}&\dot{\hat{\Theta}}(t), \label{thetahatdot}
\end{align}
by noting that $u(x,t)=\partial_t \bar{u}(x,t)=\alpha(x,t)-\partial_t \beta(x,t)$ for $(x, t) \in (0, D) \times [0, \infty)$ and recalling the cascades \eqref{Thetadot}--\eqref{thetadot}, \eqref{Sperturb}--\eqref{betax0} and the solution of $\beta(x,t)$ in \eqref{betaxtsoluti}.

\subsection{Parameter estimator design and error dynamics}
We consider the following backstepping transformation
\begin{equation}
    w(x,t)={}u(x,t)+\int_0^x q(x,r)u(r,t)dr+\gamma(x)\tilde{\theta}(t), \label{backstepheat}
\end{equation}
with the gain kernels $q(x,r)=\bar{k}(x-r)$ and $\gamma(x)=\bar{k}x$, $\bar{k}>0$, which transform the cascade \eqref{Thetahatdot}--\eqref{thetahatdot} into the target system
\begin{align}
    \dot{\tilde{\theta}}(t)={}&-\bar{k} \tilde{\theta}(t)+\partial_x w(0,t), \\
    \partial_t w(x,t)={}&\partial_{xx} w(x,t), \quad x\in (0, D), \\
    w(0,t)={}&0, \\
    \partial_x w(D,t)={}&0, \label{wxDt}
\end{align}
by noting $\dot{\tilde{\theta}}(t)=\dot{\hat{\theta}}(t)$ from \eqref{errorsheat}. Taking the derivative of \eqref{backstepheat} with respect to $x$, setting $x=D$ in the resulting expression, and recalling \eqref{wxDt}, we derive the update law as 
\begin{align}
    \dot{\hat{\Theta}}(t)={}&-\bar{k}\tilde{\theta}(t)-\bar{k} \int_0^D  u(r,t)dr, \nonumber  \\
    ={}&-\bar{k}\tilde{\theta}(t)-\bar{k} \big( \hat{\Theta}(t)-\hat{\theta}(t)\big). \label{Updeheat}
\end{align}
In \eqref{Updeheat}, we apply the property that $u(x,t)=\partial_t \bar{u}(x,t)=\partial_{xx} \bar{u}(x,t)$, and use \eqref{errorsheat}, \eqref{errheat1}, and \eqref{errheat2}. However, \eqref{Updeheat} requires direct measurement of $\tilde{\theta}(t)$. As discussed in the text following \eqref{UpdLaw}, we can redesign \eqref{Updeheat} as follows
\begin{equation}
    \dot{\hat{\Theta}}(t)={}-k G(t)-k \hat{H}(t) \Big( \hat{\Theta}(t)-\theta(t)+ e^{-\lambda t}a\sin(\omega t)\Big) \label{updHeat}
\end{equation}
by recalling \eqref{TheHat}, using the feedback of $\theta(t)$, and choosing $\bar{k}=kH$ where $H > 0$ represents the unknown Hessian  and $k > 0$ is the user-defined gain. The estimates $G(t)$, $\hat{H}(t)$, and $\eta$-dynamics have the same form as in \eqref{Gdef}, \eqref{Hdef} and \eqref{etadyn}, respectively. Additionally, the parameters should satisfy the following conditions
\begin{equation}
    \lambda<\min\left\{\frac{\omega_h}{2}, \frac{\pi^2}{4D^2}\right\}, \qquad k> \frac{\lambda}{H}. \label{paramheat}
\end{equation}
Using \eqref{etadyn}, \eqref{etatilde}, \eqref{TheHat}, \eqref{errorsheat}--\eqref{errheat2}, and \eqref{updHeat}, we write the following closed-loop error system
\begin{align}
    \tilde{\theta}(t)={}&\partial_x \bar{u}(0,t), \label{closedheat1} \\
    \partial_t \bar{u}(x,t)={}&\partial_{xx} \bar{u}(x,t), \\
    \bar{u}(0,t)={}&0, \\
    \partial_x \bar{u}(D,t)={}&\tilde{\Theta}(t), \\
    \dot{\tilde{\Theta}}(t)={}&-kG(t)-k\hat{H}(t)\big(\tilde{\Theta}(t)-\tilde{\theta}(t)\big), \\
    \dot{\tilde{\eta}}(t)={}& -\omega_h \tilde{\eta}(t)+\omega_h  \left(y(t)-y^*\right), \label{closedheat2} 
\end{align}
where $G(t)$ and $\hat{H}(t)$ are rewritten in the same form as \eqref{Grewrite} and \eqref{Hrewrite}, respectively. Using \eqref{TheHat}, \eqref{errorsheat}, the output \eqref{outheat} is rewritten as 
\begin{equation}
    y(t)={}y^*+\frac{H}{2}(\tilde{\theta}(t)+a\sin(\omega t))^2.
\end{equation}

\subsection{Stability analysis}
The main theorem is stated as follows:
\begin{theorem} \label{theodiff}
Let Assumption 1 holds and parameters satisfy \eqref{paramheat}. Then,
there exists $\bar{\omega}$ and for any $\omega > \bar{\omega}$, the closed-loop system \eqref{closedheat1}--\eqref{closedheat2} exponentially stable at the origin in the sense of the norm
\begin{align}
    \left( ||\bar{u}(\cdot, t)||^2+|\tilde{\Theta}(t)|^2+\left|\tilde{\eta}(t)\right|^2 \right)^{1/2}. \label{theorem2norm}
\end{align}
Furthermore,  
the input $\theta(t)$ and output $y(t)$ exponentially converge to $\theta^*$ and $y^*$, respectively.
\end{theorem}

\begin{proof}
Let us proceed through the proof step by step.

\textbf{Step 1: State transformation.}  Let us consider the following transformations    
\begin{align}
    \tilde{\theta}_f(t)={}&e^{\lambda t} \partial_x \bar{u}(0,t), \label{transheat1} \\
    u_f(x,t)={}&e^{\lambda t} \bar{u}(x,t), \\
    \tilde{\Theta}_f(t)={}&e^{\lambda t}\tilde{\Theta}(t), \\
    \tilde{\eta}_f(t)={}&e^{2\lambda t}\tilde{\eta}(t), \label{transheat2}
\end{align}
which transform \eqref{closedheat1}--\eqref{closedheat2} to the following system
\begin{align}
    \tilde{\theta}_f(t)={}&\partial_x \bar{u}_f(0,t), \label{transfhpde0}  \\
    \partial_t \bar{u}_f(x,t)={}&\partial_{xx} \bar{u}_f(x,t)+\lambda u_f(x,t), \label{transfhpde1} \\
    \bar{u}_f(0,t)={}&0, \\
    \partial_x \bar{u}_f(D,t)={}&\tilde{\Theta}_f(t), \\
    \dot{\tilde{\Theta}}_f(t)={}&\lambda\tilde{\theta}_f(t)  -k \frac{2}{a}\sin(\omega t)\frac{H}{2}\Big[\Big(\tilde{\theta}_f(t)+a \sin(\omega t)\Big)^2 \nonumber \\
    -{}&\tilde{\eta}_f(t)\Big]-k\left(-\frac{8}{a^2}\cos(2\omega t)\right)\frac{H}{2}\Big[\Big(\tilde{\theta}_f(t) \nonumber \\
    +&{}a \sin(\omega t)\Big)^2-\tilde{\eta}_f(t)\Big]  \Big(\tilde{\Theta}_f(t)-\tilde{\theta}_f(t)\Big), \\
    \dot{\tilde{\eta}}_f(t)={}& -(\omega_h-2\lambda) \tilde{\eta}_f(t)\nonumber \\
    &+\omega_h \frac{H}{2}\big[\tilde{\theta}_f(t)+a \sin(\omega t)\big]^2. \label{transfhpde2}
\end{align}

\textbf{Step 2: Averaging operation.} 
The average of the transformed system \eqref{transfhpde0}--\eqref{transfhpde2} over the period $\Pi=2\pi/\omega$ is given by
\begin{align}
    \tilde{\theta}^{\rm av}_f(t)={}&\partial_x \bar{u}^{\rm av}_f(0,t), \label{Thefavheat} \\
    \partial_t \bar{u}^{\rm av}_{f}(x,t)={}&\partial_{xx} \bar{u}^{\rm av}_{f}(x,t)+\lambda u^{\rm av}_f(x,t), \label{uavfpde1} \\
    \bar{u}^{\rm av}_f(0,t)={}&0, \\
    \partial_x \bar{u}^{\rm av}_f(D,t)={}&\tilde{\Theta}^{\rm av}_f(t), \label{uavfpde2}\\ 
    \dot{\tilde{\Theta}}^{\rm av}_f(t)={}&-(kH-\lambda)\tilde{\Theta}_f^{\rm av}(t),  \label{thetildotavheat}  \\
    \dot{\tilde{\eta}}^{\rm av}_f(t)={}&-(\omega_h-2\lambda) \tilde{\eta}^{\rm av}_f(t) \nonumber \\
    &+\omega_h \frac{H}{2} \left( (\tilde{\theta}^{\rm av}_f(t))^2+ \frac{a^2}{2} \right), \label{etafavheat}
\end{align}
where $\tilde{\theta}^{\rm av}_f(t)$, $\bar{u}^{\rm av}_{f}(\cdot,t)$, $\tilde{\Theta}_f^{\rm av}(t)$, and $\tilde{\eta}^{\rm av}_f(t)$ denote the average versions
of the states $\tilde{\theta}_f(t)$, $\bar{u}_{f}(\cdot,t)$, $\tilde{\Theta}_f(t)$, and $\tilde{\eta}_f(t)$, respectively.

\textbf{Step 3: Stability of average system.} The solution to \eqref{thetildotavheat} is given by 
$
    \tilde{\Theta}_f^{\rm av}(t)=\tilde{\Theta}_f^{\rm av}(0)e^{-(kH-\lambda)t}.
$
Then, using the method of separation of variables, we obtain the exact solution to the reaction-diffusion equation \eqref{uavfpde1}--\eqref{uavfpde2} as 
\begin{align}
    \bar{u}^{\rm av}_{f}(x,t)=&\left(\frac{\tilde{\Theta}_f^{\rm av}(0)}{\sqrt{kH}\cos(\sqrt{kH}D)}\right)  \sin\left(\sqrt{kH}x\right)e^{-(kH-\lambda)t} \nonumber \\
    &+\sum_{n=1}^{\infty}e^{\left(\lambda-\frac{\pi^2 (2n-1)^2}{4D^2}\right)t}\sin\left(\frac{\pi (2n-1) x}{2D}\right) M_n, \label{ubarexsol}
\end{align}
where $M_n=\frac{1}{D}\int_0^{2D} \bar{u}_f^{\rm av}(x,0) \sin\left(\frac{\pi (2n-1) x}{2D}\right)dx$.
Thus, the $(\tilde{\Theta}_f^{\rm av}, \bar{u}^{\rm av}_f)$-system is exponentially stable at the origin for $kH>\lambda$ and $\lambda<\frac{\pi^2}{4D^2}$. We establish the exponential convergence of $\tilde{\theta}_f^{\rm av}(t)$ to zero recalling \eqref{Thefavheat} and using \eqref{ubarexsol}. Using this fact, it is trivial to show that $\tilde{\eta}_f^{\rm av}(t)$ of \eqref{etafavheat}  exponentially converges to $\frac{\omega_hH a^2}{4(\omega_h-2\lambda)}$ for $\omega_h>2\lambda$.

\textbf{Step 4: Invoking averaging theorem.} 
Applying the averaging theorem \cite{hale1990averaging}, we establish that there exists $\bar{\omega}$ and for any $\omega > \bar{\omega}$, the transformed system \eqref{transfhpde0}--\eqref{transfhpde2} with states $(\bar{u}_f(\cdot,t), \tilde{\Theta}_f(t), \tilde{\eta}_f(t))$ has a unique  exponentially stable periodic solution $(\bar{u}^{\Pi}_f(\cdot, t), \tilde{\Theta}_f^{\Pi}(t), \tilde{\eta}_f^{\Pi}(t))$ of period $\Pi=2\pi/\omega$ and this solution satisfies
\begin{equation}
    \bigg(||\bar{u}^{\Pi}_f(\cdot, t)||^2+|\tilde{\Theta}_f^{\Pi}(t)|^2+\bigg|\tilde{\eta}_f^{\Pi}(t)-\frac{\omega_h H a^2}{4(\omega_h-2\lambda)}\bigg|^2\bigg)^{\frac{1}{2}} \leq \mathcal{O}\Big(\frac{1}{\omega}\Big) \label{transheatO}  
\end{equation}
for all $t \geq 0$.
Considering \eqref{transheatO} and recalling the
transformations \eqref{transheat1}--\eqref{transheat2}, we deduce that the original error system
\eqref{closedheat1}--\eqref{closedheat2} with states $\bar{u}(\cdot,t), \tilde{\Theta}(t), \tilde{\eta}(t)$ has a unique solution and is exponentially stable at the origin in the sense of the norm \eqref{theorem2norm}. In addition, $\tilde{\theta}(t)=\partial_x \bar{u}(0,t)$ exponentially converges to zero. 

\begin{figure}[t]
    \centering
    \includegraphics[width=\columnwidth]{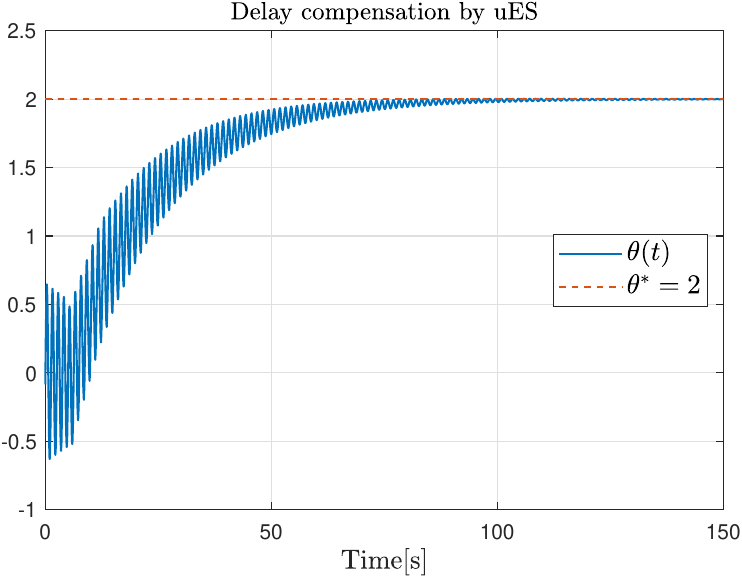} 
    \caption{The trajectory of input $\theta$ resulting from the application of the delay-compensated uES \eqref{delayuES} in the presence of a delay of $D=5$ seconds.}
    \label{DelayuESFig}
\end{figure}

\textbf{Step 5: Convergence to extremum.} 
Taking into account the results in Step 4 and recalling from \eqref{TheHat}, \eqref{errorsheat} that
\begin{align}
    \theta(t)={}e^{-\lambda t} \tilde{\theta}_f(t)+\theta^*+e^{-\lambda t}a\sin(\omega t),
\end{align}
we conclude the exponential convergence of $\theta(t)$ to $\theta^*$. 
Then, we establish the convergence of the output $y(t)$ to $y^*$ from \eqref{outheat} and complete the proof of Theorem \ref{theodiff}.
\end{proof}

\section{Numerical Simulation} \label{numsim}
In this section, we perform a numerical simulation to evaluate the performance of the developed ES algorithms. We consider the following static quadratic map
\begin{align}
    Q(\theta)={}1+(\theta-2)^2. \label{mapforsim}
\end{align}

In the first scenario, we examine a case where the map \eqref{mapforsim} is measured with a known delay of $D=5$. 
We implement the delay-compensated uES \eqref{delayuES} with parameters $k=0.03, a=0.8, \omega=5, \omega_h=1, \lambda=0.04$. All initial conditions are set to zero. As illustrated in Fig. \ref{DelayuESFig}, the delay-compensated uES algorithm effectively compensates for the delay and ensures unbiased convergence of the input $\theta$ to its optimum value $\theta^*=2$ exponentially at a rate of $\lambda$.

In the second scenario, we examine a case where the map \eqref{mapforsim} is coupled with a diffusion PDE \eqref{Thetadot}--\eqref{thetadot}, with $D=1$. 
We employ the diffusion PDE-compensated uES \eqref{updHeat} with parameters identical to those used in the delay-compensated design, and depict the result in Figure \ref{DiffuESFig}.
Our approach effectively compensates for the diffusion PDE dynamics and achieves exponential convergence to the optimum $\theta^*=2$.

\begin{figure}[t]
    \centering
    \includegraphics[width=\columnwidth]{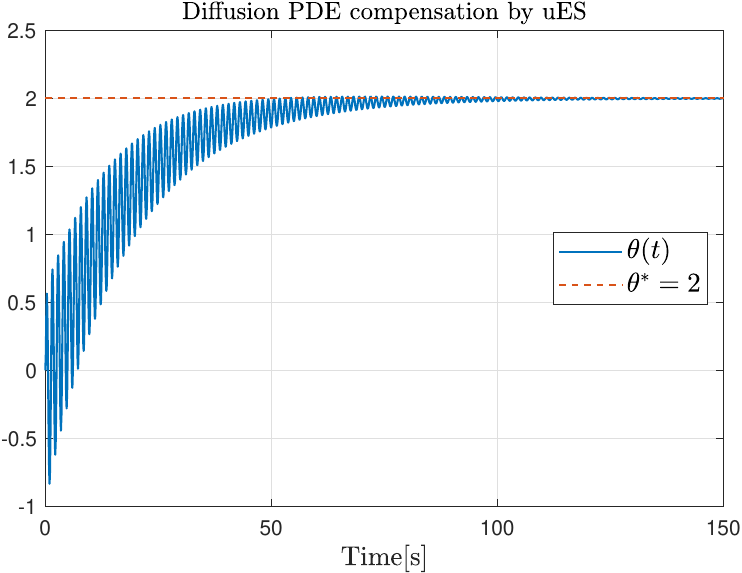} 
    \caption{The trajectory of input $\theta$ resulting from the application of the diffusion PDE-compensated uES \eqref{updHeat}.}
    \label{DiffuESFig}
\end{figure}

\section{Conclusion} \label{conclusion}
Motivated by the need to address optimization challenges in complex physical systems described by PDEs, we introduce two distinct uES algorithms designed to handle PDE dynamics. One algorithm addresses arbitrarily long input-output delays, while the other deals with diffusion PDE input dynamics. Unlike existing designs that use constant amplitude dither signals, we use additive and multiplicative dither signals with exponentially decaying and growing amplitudes, respectively. By carefully adjusting the design parameters, we achieve unbiased convergence at a user-defined exponential rate. A promising direction for future research would be to explore the application of this approach for seeking the critical heat flux in pool boiling systems.


%%%%%%%%%%%%%%%%%%%%%%%%%%%%%%%%%%%%%%%%%%%%%%%%%%%%%%%%%%%%%%%%%%%%%%%%%%%%%%%%

%%%%%%%%%%%%%%%%%%%%%%%%%%%%%%%%%%%%%%%%%%%%%%%%%%%%%%%%%%%%%%%%%%%%%%%%%%%%%%%%

%%%%%%%%%%%%%%%%%%%%%%%%%%%%%%%%%%%%%%%%%%%%%%%%%%%%%%%%%%%%%%%%%%%%%%%%%%%%%%%%
\section*{APPENDIX}

\subsection{Proof of Lemma \ref{lemS}}
Following the PDE-based motion planning technique presented in \cite[Ch. 12]{krstic2008boundary}, we seek the solution in the following form
\begin{align}
    \beta(x,t)={}&\sum_{k=0}^{\infty} c_k(t) \frac{x^k}{k!}, \label{betaxtsol}
\end{align}
where $c_k(t)$ are the time-varying coefficients. To determine these coefficients, we first use the boundary conditions \eqref{beta0} and \eqref{betax0}, and obtain
\begin{align}
    c_0(t)={}&0, \label{c0} \\
    c_1(t)={}&e^{-\lambda t}a\sin(\omega t)={}\text{Im}\left\{a e^{(-\lambda+j\omega)t}\right\}. \label{c1}
\end{align}
Next, we substitute \eqref{betaxtsol} into \eqref{betatxx}, resulting in the recursive relationship $c_{k+2}(t)=\dot{c}_k(t)$. From \eqref{c0} and \eqref{c1}, we get
\begin{align}
    c_{2k}(t)={}&0, \\
    c_{2k+1}(t)={}&\text{Im}\left\{a(-\lambda+j\omega)^{k} e^{(-\lambda+j\omega)t}\right\}. \label{c2k1}
\end{align}
Substituting \eqref{c2k1} into \eqref{betaxtsol}, we get
\begin{align}
    \beta(x,t)={}&\text{Im}\left\{\sum_{k=0}^{\infty} a(-\lambda+j\omega)^{k}\frac{x^{2k+1}}{(2k+1)!} e^{(-\lambda+j\omega)t}\right\}, \nonumber \\
    ={}&\text{Im}\left\{\sum_{k=0}^{\infty} \frac{a\big(\sqrt{(-\lambda+j\omega)}x\big)^{2k+1}}{\sqrt{(-\lambda+j\omega)}(2k+1)!}e^{(-\lambda+j\omega)t} \right\}, \nonumber \\
    ={}&\text{Im}\left\{ \frac{a}{\sqrt{-\lambda+j\omega}} \sinh(\sqrt{(-\lambda+j\omega)}x) e^{(-\lambda+j\omega)t} \right\}, \nonumber \\
    ={}&\text{Im}\Big\{ \left(e^{px-\lambda t+j(qx+\omega t)}-e^{-px-\lambda t+j(-qx+\omega t)} \right)\nonumber \\
    &+a(p-jq)/(2p^2+2q^2) \Big\},
\end{align}
where $p+jq=\sqrt{(-\lambda+j\omega)}$ and $p$ and $q$ are  defined after \eqref{Sexplicit}.
Finally, we arrive at
\begin{align}
    \beta(x,t)={}&\big[\left(p\sin(\omega t+qx)-q\cos(\omega t+qx)\right)e^{px-\lambda t} \nonumber \\
    &-\left(p\sin(\omega t-qx)-q\cos(\omega t-qx)\right)e^{-px-\lambda t}\big] \nonumber \\
    &\times a/(2p^2+2q^2). \label{betaxtsoluti}
\end{align}
Differentiating \eqref{betaxtsoluti} with respect to $x$ and setting $x = D$, we derive \eqref{Sexplicit}. \hfill $\blacksquare$

%%%%%%%%%%%%%%%%%%%%%%%%%%%%%%%%%%%%%%%%%%%%%%%%%%%%%%%%%%%%%%%%%%%%%%%%%%%%%%%%


\begin{thebibliography}{10}

\bibitem{aarsnes2019extremum}
U.~J.~F. Aarsnes, O.~M. Aamo, and M.~Krstic.
\newblock Extremum seeking for real-time optimal drilling control.
\newblock In {\em 2019 American Control Conference (ACC)}, pages 5222--5227. IEEE, 2019.

\bibitem{abbasgholinejad2023extremum}
E.~Abbasgholinejad, H.~Deng, J.~Gamble, J.~N. Kutz, E.~Nielsen, N.~Pisenti, and N.~Xie.
\newblock Extremum seeking control of quantum gates.
\newblock In {\em 2023 IEEE International Conference on Quantum Computing and Engineering (QCE)}, volume~2, pages 227--231. IEEE, 2023.

\bibitem{alessandri2020stabilization}
A.~Alessandri, P.~Bagnerini, R.~Cianci, S.~Donnarumma, and A.~Taddeo.
\newblock Stabilization of diffusive systems using backstepping and the circle criterion.
\newblock {\em International Journal of Heat and Mass Transfer}, 149:119132, 2020.

\bibitem{bovskovic2002backstepping}
D.~M. Bo{\v{s}}kovi{\'c} and M.~Krsti{\'c}.
\newblock Backstepping control of chemical tubular reactors.
\newblock {\em Computers \& chemical engineering}, 26(7-8):1077--1085, 2002.

\bibitem{feiling2018gradient}
J.~Feiling, S.~Koga, M.~Krsti{\'c}, and T.~R. Oliveira.
\newblock Gradient extremum seeking for static maps with actuation dynamics governed by diffusion pdes.
\newblock {\em Automatica}, 95:197--206, 2018.

\bibitem{hale1990averaging}
J.~K. Hale and S.~V. Lunel.
\newblock Averaging in infinite dimensions.
\newblock {\em The Journal of integral equations and applications}, pages 463--494, 1990.

\bibitem{hudon2008adaptive}
N.~Hudon, M.~Guay, M.~Perrier, and D.~Dochain.
\newblock Adaptive extremum-seeking control of convection-reaction distributed reactor with limited actuation.
\newblock {\em Computers \& Chemical Engineering}, 32(12):2994--3001, 2008.

\bibitem{koga2020laser}
S.~Koga, M.~Krstic, and J.~Beaman.
\newblock Laser sintering control for metal additive manufacturing by pde backstepping.
\newblock {\em IEEE Transactions on Control Systems Technology}, 28(5):1928--1939, 2020.

\bibitem{krstic2009delay}
M.~Krstic.
\newblock {\em Delay compensation for nonlinear, adaptive, and PDE systems}.
\newblock Springer, 2009.

\bibitem{krstic2008boundary}
M.~Krstic and A.~Smyshlyaev.
\newblock {\em Boundary control of PDEs: A course on backstepping designs}.
\newblock SIAM, 2008.

\bibitem{kumar2020extremum}
S.~Kumar, M.~R. Zwall, E.~A. Bol{\'\i}var-Nieto, R.~D. Gregg, and N.~Gans.
\newblock Extremum seeking control for stiffness auto-tuning of a quasi-passive ankle exoskeleton.
\newblock {\em IEEE Robotics and Automation Letters}, 5(3):4604--4611, 2020.

\bibitem{leblanc1922electrification}
M.~Leblanc.
\newblock Sur l'electrification des chemins de fer au moyen de courants alternatifs de frequence elevee.
\newblock {\em Revue g{\'e}n{\'e}rale de l'{\'e}lectricit{\'e}}, 12(8):275--277, 1922.

\bibitem{moura2012pde}
S.~J. Moura, N.~A. Chaturvedi, and M.~Krstic.
\newblock Pde estimation techniques for advanced battery management systems—part i: Soc estimation.
\newblock In {\em 2012 American Control Conference (ACC)}, pages 559--565. IEEE, 2012.

\bibitem{oliveira2021extremum2}
T.~R. Oliveira and M.~Krstic.
\newblock Extremum seeking boundary control for pde--pde cascades.
\newblock {\em Systems \& Control Letters}, 155:105004, 2021.

\bibitem{oliveira2021extremum}
T.~R. Oliveira and M.~Krstic.
\newblock Extremum seeking feedback with wave partial differential equation compensation.
\newblock {\em Journal of Dynamic Systems, Measurement, and Control}, 143(4):041002, 2021.

\bibitem{oliveira2022extremum}
T.~R. Oliveira and M.~Krstic.
\newblock {\em Extremum seeking through delays and PDEs}.
\newblock SIAM, 2022.

\bibitem{oliveira2016extremum}
T.~R. Oliveira, M.~Krsti{\'c}, and D.~Tsubakino.
\newblock Extremum seeking for static maps with delays.
\newblock {\em IEEE Transactions on Automatic Control}, 62(4):1911--1926, 2016.

\bibitem{scheinker2024100}
A.~Scheinker.
\newblock 100 years of extremum seeking: A survey.
\newblock {\em Automatica}, 161:111481, 2024.

\bibitem{yilmaz2023press}
C.~T. Yilmaz, M.~Diagne, and M.~Krstic.
\newblock Exponential and prescribed-time extremum seeking with unbiased convergence.
\newblock {\em arXiv preprint arXiv:2401.00300}, 2023.

\bibitem{yilmaz2023exponential}
C.~T. Yilmaz, M.~Diagne, and M.~Krstic.
\newblock Exponential extremum seeking with unbiased convergence.
\newblock In {\em 2023 62nd IEEE Conference on Decision and Control (CDC)}, pages 6749--6754. IEEE, 2023.

\bibitem{yilmaz2024perfect}
C.~T. Yilmaz, M.~Diagne, and M.~Krstic.
\newblock Perfect tracking of time-varying optimum by extremum seeking.
\newblock {\em arXiv preprint arXiv:2402.14178}, 2024.

\end{thebibliography}
\end{document}